\documentclass[11pt]{amsart}

\usepackage{amsmath,amssymb,amsthm}
\usepackage{bbm}
\usepackage{upgreek}
\usepackage{amsaddr}
\usepackage{mathtools}
\usepackage{geometry}
\usepackage{graphicx}
\usepackage{booktabs}
\usepackage{microtype}
\usepackage{hyperref}
\usepackage{enumerate}
\usepackage{xcolor}

\graphicspath{{figures/}}
\hypersetup{
  colorlinks=true,
  linkcolor=blue!45!black,
  citecolor=green!35!black,
  urlcolor=blue!55!black
}

\newtheorem{theorem}{Theorem}[section]
\newtheorem{proposition}[theorem]{Proposition}
\newtheorem{lemma}[theorem]{Lemma}
\newtheorem{corollary}[theorem]{Corollary}
\theoremstyle{definition}
\newtheorem{definition}[theorem]{Definition}

\theoremstyle{remark}
\newtheorem{remark}[theorem]{Remark}

\newcommand{\R}{\mathbb{R}}
\newcommand{\RP}{\mathbb{RP}}
\newcommand{\E}{\mathrm{E}}

\newcommand{\CSS}{\operatorname{CSS}}
\newcommand{\ord}{\operatorname{ord}}
\newcommand{\dd}{\,\mathrm{d}}

\newcommand{\Aeq}{\mathcal A}

\newcommand{\doi}[1]{\href{https://doi.org/#1}
{\nolinkurl{https://doi.org/#1}}}

\title[Wigner caustic and centre symmetry set of frontals]
{Singularities of the Wigner Caustic and the Centre Symmetry Set of Frontal Curves in the Euclidean Plane}

\author{Micha\l{} Zwierzy\'nski}

\email{Michal.Zwierzynski@pw.edu.pl}
\email{ORCID: 0000-0002-9627-1563}

\address{Warsaw University of Technology\\
Faculty of Mathematics and Information Science\\
ul. Koszykowa 75\\
00-662 Warsaw, Poland}



\begin{document}

\begin{abstract}
Motivated by the~role of the~Wigner caustic in semiclassical phase-space analysis, we extend it, together with the~centre symmetry set, from regular planar curves to cooriented frontals.  For an~angularly regular parallel pair, the~signed speed of the~Wigner caustic is one half of the~difference of the~extended signed radii of curvature, while singular points of the~finite centre symmetry set are the~critical points of their projective ratio.  These formulas yield criteria and explicit invariants for ordinary and higher cusps, an~order-lowering relation between the~two constructions, and a~multiplicity-weighted extension of the~classical cusp-count inequality for strictly convex ovals.

We also analyse parallel pairs containing a~singular frontal which is not a~front.  A~$5/2$-cusp produces either a~$5/2$- or, at a~signed-radius resonance, a~$5/3$-cusp on the~Wigner caustic.  The~opposite resonance sends the~centre symmetry set to infinity.  We obtain the~corresponding transfer results for a~$5/3$-cusp and give a~projective completion which resolves vanishing denominators and simultaneous finite-order zeros of the~two signed radii.
\end{abstract}

\maketitle

\tableofcontents

\section{Introduction}

\noindent The~Wigner caustic of a~regular planar curve is the~locus of midpoints of chords joining points with parallel tangent lines.  Its name originates in semiclassical phase-space analysis.  In the~one-dimensional semiclassical setting, the~classical correspondence of a~pure quantum state is represented by a~smooth Lagrangian curve in phase space, and Berry showed that the~associated Wigner function exhibits enhanced values near both this curve and an~additional singular locus, called the~Wigner caustic or Wigner catastrophe \cite{Berry}.  The~local singularities and global branch structure of this locus are therefore relevant to Wigner-function dynamics and to the~breakdown of the~quantum--classical correspondence.  In our earlier work \cite{DomitrzZwier}, we developed the~global branch structure of the~Wigner caustic of closed planar curves and discussed its role in models of quantum dynamics in optical-lattice potentials.  More generally, affine equidistants record fixed affine division ratios on the~same parallel-tangent chords.  Their local singularities and global geometry have been studied from several viewpoints \cite{DomitrzRiosRuas,DomitrzZwierSingular,MillerZwier}.

The~centre symmetry set, introduced in \cite{Janeczko1996}, is the~envelope of the~family of parallel-tangent chords and may equivalently be viewed as the~locus swept out by singular points of affine equidistants as the~division ratio varies \cite{DomitrzRiosGCS,GiblinHoltom,GiblinReeveCSS,GiblinZakalyukin}.  For generic smooth strictly convex closed curves, the~number of cusps of the~Wigner caustic does not exceed the~number of cusps of the~centre symmetry set \cite[Theorem~7.2]{DomitrzRiosGCS}.

Both constructions occur in several neighbouring settings.  Wigner caustics, affine equidistants, and global centre symmetry sets have Lagrangian and higher-dimensional analogues (see \cite{DomitrzManoelRios,DomitrzRiosGCS,GiblinReeve}, and the literature therein). The~centre-chord construction also produces singular improper affine spheres \cite{CraizerDomitrzRios}.  Algebraic versions of the~symmetry defect have been studied for varieties and, more explicitly, for generic algebraic curves \cite{DiasFarnikJelonek,JaneczkoJelonekRuas}.  Their visual and artistic features are discussed in \cite{DanielewskaEtAl}.

In convex geometry, the~Wigner caustic is closely related to the~middle hedgehog and to reflection, offset, and projective-hedgehog constructions \cite{Rochera2022,SchneiderReflections,SchneiderMiddle}.  Its oriented area enters refinements and extensions of isoperimetric inequalities for ovals, rosettes, and in normed planes \cite{dosSantosCraizer2022,Zhang2021,Zwier2016,ZwierRosettes,ZwierCWMS}. Fourier and mixed-volume generalisations appear in \cite{SafarewiczZwier2026,ZwierMixedVolumes}.  Discrete Wigner caustics (known also as the \emph{area evolute}) and centre symmetry sets for polygons with parallel opposite sides were developed in \cite{CraizerTeixeiraSilva,KonicerEtAl2026}.

A~frontal curve may have singular points but still carries a~smooth tangent-line field.  This makes both chord constructions meaningful beyond regular curves.  In recent years, frontals have attracted substantial interest in differential geometry, singularity theory, contact geometry, and related areas.  Representative developments include global curvature results for singular curves \cite{HondaTanakaYamauchi}, parallels and evolutes of planar frontals \cite{TeramotoParallels}, deformation theory and singularities of frontal surfaces \cite{MunozNunoOsetDeformations,MunozNunoOsetSurfaces}, and dualities for evolving developable frontals \cite{YangLiLopez} (see also \cite{DomitrzZwierGaussBonnet,IshikawaRecognition,SajiUmeharaYamadaFronts,SajiUmeharaYamadaCTB} and the~literature therein).  For planar frontals, the~basic differential data are a~signed speed and a~tangent angle \cite{FukunagaTakahashi,UmeharaSajiYamada}. These data have recently proved effective in the~study of pedal-curve singularities \cite{LiPei,Teramoto}.  Our purpose is to show that they also give transparent formulas for the~Wigner caustic and the~centre symmetry set, while the~bilocal nature of the~two constructions creates additional phenomena.

There are two levels of degeneracy.  First, a~frontal may be singular while its tangent-angle map is regular.  Ordinary $3/2$-cusps and the~front-type $4/3$- and $5/4$-cusps belong to this class.  A~common tangent angle can then be used as a~parameter for both germs in a~parallel pair.  The~usual signed radius of curvature extends smoothly across their singular points, with value zero.  We prove that the~difference and the~projective ratio of the~two extended radii control, respectively, the~Wigner caustic and the~centre symmetry set.  These formulas yield criteria for ordinary and higher cusps, cuspidal-curvature formulas, and an~order-lowering principle.  For strictly convex ovals, they also give a~finite-order extension of the~classical cusp-count comparison, including multiplicity-weighted counts.

Second, at a~singular frontal which is not a~front both the~signed speed and the~angular speed vanish.  A~common angular parameter is no longer available on that germ.  If the~other germ in the~parallel pair has a~regular tangent-angle map, the~parallel-pair locus is nevertheless smooth.  For a~$5/2$-cusp with nonzero bias, direct formulas reveal two opposite signed-radius resonances: one changes the~induced Wigner singularity into a~$5/3$-cusp, whereas the~other places the~centre symmetry set at infinity.  We also determine the~corresponding transfer behaviour for a~$5/3$-cusp.  These phenomena motivate a~projective completion of the~centre symmetry set.

Except for the~oval cusp-count theorem in Section~\ref{sec:relations}, the~results are local.  The~two germs may belong to the~same frontal, provided that their parameters are distinct, or to two different frontals.  We exclude coincident chord endpoints and tangent chords whenever an~ordinary affine envelope is discussed.  These hypotheses separate the~new frontal phenomena from degeneracies of the~parallel-pair locus already present for regular curves with inflections.

\section{Frontal curves and cusp criteria}

\noindent Let $\gamma\colon(\R,0)\to\R^2$ be a~smooth map germ.  It is a~\emph{cooriented frontal} if it admits a~smooth unit normal field $\mathbbm{n}$ such that $\gamma'(t)\cdot\mathbbm{n}(t)=0$, where $\cdot$ denotes the standard dot product.  Choose the~unit tangent field $\mathbbm{e}$ so that $(\mathbbm{e},\mathbbm{n})$ is positively oriented.  Locally there are smooth functions $\ell$ and $\theta$ satisfying
\begin{equation}\label{eq:Legendre-data}
 \gamma'(t)=\ell(t)\mathbbm{e}(t),\qquad
 \mathbbm{e}(t)=(\cos\theta(t),\sin\theta(t)),\qquad
 \mathbbm{n}(t)=(-\sin\theta(t),\cos\theta(t)).
\end{equation}
Consequently,
\begin{equation}\label{eq:Frenet}
 \mathbbm{e}'(t)=\theta'(t)\mathbbm{n}(t),\qquad
 \mathbbm{n}'(t)=-\theta'(t)\mathbbm{e}(t).
\end{equation}
The~pair $(\gamma,\mathbbm{n})$ is an~immersion if and only if
\begin{equation}\label{eq:front-criterion}
 (\ell(t),\theta'(t))\ne(0,0).
\end{equation}
In this case $\gamma$ is called a~\emph{front}.  We call a~point \emph{angularly regular} if $\theta'(t)\ne0$.

At an~angularly regular point, put $u=\theta(t)-\theta(0)$ and denote the~inverse change of parameter by $t=t(u)$.  Then
\begin{equation}\label{eq:extended-radius}
 \frac{\dd\gamma}{\dd u}
 =\rho(u)\mathbbm{e}(u),
 \qquad
 \rho(u)=
 \frac{\ell(t(u))}{\theta'(t(u))}.
\end{equation}
At a~regular point, the~signed curvature is $\kappa=\theta'/\ell$, so that $\rho=1/\kappa$ is the~signed radius of curvature.  At a~singular point of an~angularly regular front, the~function $\rho$ remains smooth and vanishes.  We therefore call it the~\emph{extended signed radius of curvature}.

Two map germs $f,g\colon(\R,0)\to\R^2$ are $\Aeq$-equivalent if they differ by diffeomorphism germs in the~source and target (see, for instance, \cite{ArnoldGuseinZadeVarchenko,BruceGiblin} for background).  We use the~following standard terminology:
\begin{align*}
 3/2\text{-cusp}&\!:\quad t\mapsto(t^2,t^3),&
 4/3\text{-cusp}&\!:\quad t\mapsto(t^3,t^4),\\
 5/4\text{-cusp}&\!:\quad t\mapsto(t^4,t^5),&
 (5/4;\mathord\pm7)\text{-cusp}&\!:\quad t\mapsto(t^4,t^5\pm t^7),\\
 5/2\text{-cusp}&\!:\quad t\mapsto(t^2,t^5),&
 5/3\text{-cusp}&\!:\quad t\mapsto(t^3,t^5).
\end{align*}

We shall use the~following general jet criterion (see \cite[Theorems~4.1 and~4.13]{Matsushita} and \cite[Fact~2.4(3)--(4)]{Teramoto}).

\begin{proposition}\label{prop:five-four-jet-criterion}
Let $f\colon(\R,0)\to(\R^2,f(0))$ be a~smooth map germ satisfying $f'(0)=0$.  Then $f$ is $C^1$-equivalent to a~$5/4$-cusp if and only if
\begin{equation}\label{eq:five-four-C1-jet-criterion}
 f''(0)=f'''(0)=0,
 \qquad
 \det\bigl(f^{(4)}(0),f^{(5)}(0)\bigr)\ne0.
\end{equation}
Under these conditions, define $\mathcal P(f)$ as
\begin{align}\label{eq:five-four-jet-invariant}
 -77\det\bigl(f^{(4)},f^{(6)}\bigr)^2
+105\det\bigl(f^{(4)},f^{(5)}\bigr)
          \det\bigl(f^{(5)},f^{(6)}\bigr)
+60\det\bigl(f^{(4)},f^{(5)}\bigr)
         \det\bigl(f^{(4)},f^{(7)}\bigr),
\end{align}
where all derivatives are evaluated at $0$.  Then $f$ is a~$5/4$-cusp, a~$(5/4;+7)$-cusp, or a~$(5/4;-7)$-cusp according as $\mathcal P(f)$ is zero, positive, or negative, respectively.
\end{proposition}

\begin{proof}
The~$C^1$ assertion is \cite[Fact~2.4(3)]{Teramoto}.  The~zero criterion and the~two sign criteria are \cite[Theorem~4.1]{Matsushita} and \cite[Theorem~4.13]{Matsushita}, respectively, written in the~determinant convention adopted here.
\end{proof}

We now express the~cusp criteria in the~Legendre data.  Items~(i)--(iv) below are recalled from \cite[Proposition~2.5]{Teramoto}, while items~(v)--(vi) follow from Proposition~\ref{prop:five-four-jet-criterion}.

\begin{lemma}\label{lem:cusp-criteria}
Let a~frontal germ have the~data $(\ell,\theta)$ in \eqref{eq:Legendre-data} and assume that its component functions are not flat at the~origin.
\begin{enumerate}[(i)]
\item It is a~$3/2$-cusp if and only if $\ell(0)=0$ and $\ell'(0)\theta'(0)\ne0$.
\item It is a~$4/3$-cusp if and only if $\ell(0)=\ell'(0)=0$ and $\ell''(0)\theta'(0)\ne0$.
\item It is a~$5/2$-cusp if and only if
      \begin{equation}\label{eq:five-two-criterion}
      \ell(0)=\theta'(0)=0,\qquad \ell'(0)\ne0,
      \qquad
      \ell'(0)\theta'''(0)-\ell''(0)\theta''(0)\ne0.
      \end{equation}
\item It is a~$5/3$-cusp if and only if
      \begin{equation}\label{eq:five-three-criterion}
      \ell(0)=\ell'(0)=\theta'(0)=0,\qquad
      \ell''(0)\theta''(0)\ne0.
      \end{equation}
\item Suppose that $\theta'(0)\ne0$, and let $\rho$ be the~extended signed radius of curvature defined in \eqref{eq:extended-radius}. The~germ is $C^1$-equivalent to a~$5/4$-cusp if and only if
      \begin{equation}\label{eq:five-four-C1-criterion}
      \rho(0)=\rho'(0)=\rho''(0)=0,
      \qquad \rho'''(0)\ne0.
      \end{equation}
\item Under \eqref{eq:five-four-C1-criterion}, define
      \begin{equation}\label{eq:universal-five-four-invariant}
      \mathcal Q(\rho)=35\rho^{(4)}(0)^2
      -48\rho'''(0)\rho^{(5)}(0)
      +2400\rho'''(0)^2.
      \end{equation}
The~germ is a~$5/4$-cusp, a~$(5/4;+7)$-cusp, or a $(5/4;-7)$-cusp according as $\mathcal Q(\rho)$ is zero, positive, or negative, respectively.
\end{enumerate}
\end{lemma}

\begin{proof}
Items~(i)--(iv) are the~standard Legendre-data forms of the~cusp criteria quoted above.  We verify the~two additional assertions.  Write the~germ in the~angular parameter as $f'(u)=\rho(u)\mathbbm{e}(u)$.  Formula \eqref{eq:Frenet} shows successively that
\begin{equation*}
 f'(0)=f''(0)=f'''(0)=0
 \quad\Longleftrightarrow\quad
 \rho(0)=\rho'(0)=\rho''(0)=0.
\end{equation*}
Under these vanishings, repeated differentiation gives
\begin{align*}
 f^{(4)}(0)&=\rho'''(0)\mathbbm{e}(0),\\
 f^{(5)}(0)&=\rho^{(4)}(0)\mathbbm{e}(0)
              +4\rho'''(0)\mathbbm{n}(0),&
              \\
 f^{(6)}(0)&=\bigl(\rho^{(5)}(0)-10\rho'''(0)\bigr)\mathbbm{e}(0)
              +5\rho^{(4)}(0)\mathbbm{n}(0),\\
 f^{(7)}(0)&=\bigl(\rho^{(6)}(0)-15\rho^{(4)}(0)\bigr)\mathbbm{e}(0)
              +\bigl(6\rho^{(5)}(0)-20\rho'''(0)\bigr)\mathbbm{n}(0).
\end{align*}
In particular,
\begin{equation*}
 \det\bigl(f^{(4)}(0),f^{(5)}(0)\bigr)=4\rho'''(0)^2,
\end{equation*}
so Proposition~\ref{prop:five-four-jet-criterion} proves~(v).  Direct substitution of the~displayed derivatives into \eqref{eq:five-four-jet-invariant} gives
\begin{equation*}
 \mathcal P(f)=5\rho'''(0)^2\mathcal Q(\rho).
\end{equation*}
Since $\rho'''(0)\ne0$, the~factor $5\rho'''(0)^2$ is positive.  Thus $\mathcal P(f)$ and $\mathcal Q(\rho)$ have the~same zero and the~same sign, and Proposition~\ref{prop:five-four-jet-criterion} proves~(vi).
\end{proof}

\begin{remark}\label{rem:angular-radius}
By Lemma~\ref{lem:cusp-criteria}, a~$3/2$-cusp, a~$4/3$-cusp, and any of the~$5/4$-types listed above correspond, respectively, to a~zero of the~extended signed radius of curvature $\rho$ of order $1$, $2$, and $3$.
\end{remark}

\section{Angularly regular parallel pairs}

\begin{definition}\label{def:parallel-pair}
Let $\gamma_+$ and $\gamma_-$ be two cooriented frontal germs.  We say that their base points form a~\emph{parallel pair} if their tangent lines are parallel.  Equivalently, after a~possible reversal of one of the tangent lifts, their unit tangent vectors at the~base points satisfy $\mathbbm{e}_+(0)=-\mathbbm{e}_-(0)$. The~line segment joining $\gamma_+(0)$ and $\gamma_-(0)$ is called the \emph{associated affine chord}.
\end{definition}

Consider two cooriented frontal germs $\gamma_+$ and $\gamma_-$ at a~parallel pair.  Assume first that both germs are angularly regular.  After choosing the~signs of the~tangent lifts, the~inverse function theorem provides a~common angular parameter $u$ in which
\begin{equation}\label{eq:common-angle}
 \gamma_+'(u)=\rho_+(u)\mathbbm{e}(u),\qquad
 \gamma_-'(u)=-\rho_-(u)\mathbbm{e}(u),
\end{equation}
where the~prime denotes differentiation with respect to $u$.  Set
\begin{equation}\label{eq:radius-chord-data}
 \Delta_\rho=\rho_+-\rho_-,\qquad
 \Sigma_\rho=\rho_++\rho_-,\qquad
 \Delta_\gamma=\gamma_+-\gamma_-.
\end{equation}
We always assume $\Delta_\gamma(0)\ne0$.

\subsection{The~Wigner caustic}

\begin{definition}\label{def:Wigner-branch}
The~local Wigner caustic branch generated by the~parallel pair $(\gamma_+,\gamma_-)$ is parametrised by
\begin{equation}\label{eq:Wigner-definition}
 \E_{1/2}(u)
 =\frac{\gamma_+(u)+\gamma_-(u)}{2}.
\end{equation}
Thus, it is the~locus of midpoints of the~chords joining corresponding points of the~nearby parallel pairs.
\end{definition}

For regular curves, the~regular-point and ordinary-cusp parts of the following result are classical (see \cite{Berry,DomitrzZwierSingular}).

\begin{theorem}\label{thm:Wigner-difference}
The~Wigner caustic satisfies
\begin{equation}\label{eq:Wigner-derivative}
 \E_{1/2}'(u)=\frac{\Delta_\rho(u)}{2}\mathbbm{e}(u).
\end{equation}
It is a~front, including at every singular point.  Moreover:
\begin{enumerate}[(i)]
\item $\E_{1/2}$ is regular at $0$ if and only if $\Delta_\rho(0)\ne0$;
\item $\E_{1/2}$ is a~$3/2$-cusp if and only if $\Delta_\rho(0)=0$ and $\Delta_\rho'(0)\ne0$;
\item $\E_{1/2}$ is a~$4/3$-cusp if and only if $\Delta_\rho(0)=\Delta_\rho'(0)=0$ and $\Delta_\rho''(0)\ne0$;
\item if $\Delta_\rho$, $\Delta_\rho'$, and $\Delta_\rho''$ vanish at $0$, while $\Delta_\rho'''(0)\ne0$, then $\E_{1/2}$ is $C^1$-equivalent to a~$5/4$-cusp.
\item Under the~conditions in~(iv), define
      \begin{equation}\label{eq:QW}
      Q_{\E_{1/2}}=
      35\Delta_\rho^{(4)}(0)^2
      -48\Delta_\rho'''(0)\Delta_\rho^{(5)}(0)
      +2400\Delta_\rho'''(0)^2.
      \end{equation}
Then $\E_{1/2}$ is a~$5/4$-cusp, a~$(5/4;+7)$-cusp, or a $(5/4;-7)$-cusp according as $Q_{\E_{1/2}}$ is zero, positive, or negative, respectively.
\end{enumerate}
\end{theorem}

\begin{proof}
Formula \eqref{eq:Wigner-derivative} follows immediately from \eqref{eq:common-angle}.  The~vector field $\mathbbm{n}(u)$ is normal to $\E_{1/2}$, and $\mathbbm{n}'(u)=-\mathbbm{e}(u)\ne0$.  Hence $(\E_{1/2},\mathbbm{n})$ is an~immersion even when $\E_{1/2}'(u)=0$.  Items~(i)--(iv) follow from Lemma~\ref{lem:cusp-criteria}, applied to the~Legendre data
\begin{equation*}
 \ell_{\E_{1/2}}=\frac{\Delta_\rho}{2},
 \qquad \theta_{\E_{1/2}}(u)=u.
\end{equation*}
Under the~conditions in~(iv), direct homogeneity gives
\begin{equation*}
 \mathcal Q\!\left(\frac{\Delta_\rho}{2}\right)
 =\frac14\left(
 35\Delta_\rho^{(4)}(0)^2
 -48\Delta_\rho'''(0)\Delta_\rho^{(5)}(0)
 +2400\Delta_\rho'''(0)^2\right)
 =\frac{Q_{\E_{1/2}}}{4}.
\end{equation*}
The~factor $1/4$ is positive, so Lemma~\ref{lem:cusp-criteria}(vi) proves~(v).
\end{proof}

We next recall the~cuspidal curvature introduced in \cite{ShibaUmehara} (see also \cite{UmeharaSajiYamada}).

\begin{definition}
If $f\colon(\R,0)\to\R^2$ is an~A-type germ, that is, $f'(0)=0$ and $f''(0)\ne0$, its cuspidal curvature is
\begin{equation}\label{eq:cuspidal-curvature-definition}
 \omega_f=\frac{\det(f''(0),f'''(0))}{|f''(0)|^{5/2}}.
\end{equation}
\end{definition}

It is invariant under orientation-preserving reparametrisations and oriented Euclidean motions. Reversing either orientation reverses its sign, so $|\omega_f|$ is independent of these choices.  Its nonvanishing is equivalent to the~ordinary-cusp condition, and its magnitude is a basic Euclidean measure of the~opening of the~cusp.

\begin{proposition}\label{prop:Wigner-cuspidal-curvature}
If $\E_{1/2}$ has a~$3/2$-cusp at $0$, its cuspidal curvature is
\begin{equation}\label{eq:Wigner-cuspidal-curvature}
 \omega_{\E_{1/2}}=
 \frac{\det(\E_{1/2}''(0),\E_{1/2}'''(0))}
 {|\E_{1/2}''(0)|^{5/2}}
 =\frac{2\sqrt2}{\sqrt{|\Delta_\rho'(0)|}}.
\end{equation}
\end{proposition}

\begin{proof}
At the~cusp, direct differentiation of \eqref{eq:Wigner-derivative} gives
\begin{equation*}
 \E_{1/2}''(0)=\frac{\Delta_\rho'(0)}2\mathbbm{e}(0),\qquad
 \E_{1/2}'''(0)=\frac{\Delta_\rho''(0)}2\mathbbm{e}(0)
                 +\Delta_\rho'(0)\mathbbm{n}(0).
\end{equation*}
Hence
\begin{equation*}
 \det(\E_{1/2}''(0),\E_{1/2}'''(0))
 =\frac{\Delta_\rho'(0)^2}{2},
 \qquad |\E_{1/2}''(0)|=\frac{|\Delta_\rho'(0)|}{2},
\end{equation*}
which gives \eqref{eq:Wigner-cuspidal-curvature}.
\end{proof}

\subsection{The~centre symmetry set}

Assume in addition that the~chord is transverse to the~common tangent line,
\begin{equation}\label{eq:chord-transverse}
 \det(\Delta_\gamma(0),\mathbbm{e}(0))\ne0.
\end{equation}

\begin{definition}\label{def:CSS-branch}
The~local \emph{centre symmetry set} branch ($\CSS$) generated by the~parallel pair $(\gamma_+,\gamma_-)$ is the~envelope of the~family of chord lines parametrised by
\begin{equation}\label{eq:chord-family}
 L(u,\lambda)
 =\gamma_-(u)+\lambda\Delta_\gamma(u).
\end{equation}
Equivalently, it is the~set of critical values of the~map $L\colon(u,\lambda)\mapsto L(u,\lambda)$.
\end{definition}

When the~two germs belong to the~same frontal, the~union of these local branches over all distinct parallel pairs gives the~corresponding global Wigner caustic and centre symmetry set.

\begin{proposition}\label{prop:CSS-envelope}
Suppose that
\begin{equation}\label{eq:finite-CSS-condition}
 \Sigma_\rho(0)\ne0.
\end{equation}
The~envelope of \eqref{eq:chord-family} is
\begin{equation}\label{eq:CSS-formula}
 S(u)=\gamma_-(u)
 +\frac{\rho_-(u)}{\Sigma_\rho(u)}\Delta_\gamma(u).
\end{equation}
It satisfies
\begin{equation}\label{eq:CSS-derivative}
 S'(u)=
 \left(\frac{\rho_-}{\Sigma_\rho}\right)'\!(u)\Delta_\gamma(u)
 =\frac{\mathcal J_\rho(u)}{\Sigma_\rho(u)^2}\Delta_\gamma(u),
 \qquad
 \mathcal J_\rho=\rho_+\rho_-'-\rho_-\rho_+'.
\end{equation}
\end{proposition}

\begin{proof}
By \eqref{eq:common-angle},
\begin{equation}\label{eq:Delta-gamma-derivative}
 \Delta_\gamma'=\Sigma_\rho\mathbbm{e}.
\end{equation}
For fixed $\lambda$, differentiation of \eqref{eq:chord-family} gives $\partial_u L=-\rho_-\mathbbm{e} +\lambda\Sigma_\rho\mathbbm{e}$. At an~envelope point this vector is parallel to $\Delta_\gamma$.  By \eqref{eq:chord-transverse}, its coefficient must vanish, which gives $\lambda=\rho_-/\Sigma_\rho$.  Differentiating \eqref{eq:CSS-formula}, the~terms parallel to $\mathbbm{e}$ cancel and yield the~first equality in \eqref{eq:CSS-derivative}.  The~second follows by differentiating $\rho_-/\Sigma_\rho$.
\end{proof}

For regular curves, the~description of the~$\CSS$ as the~discriminant swept out by singular points of affine equidistants is classical (see \cite{DomitrzRiosGCS,GiblinHoltom,GiblinZakalyukin}).

\begin{corollary}\label{cor:CSS-equidistant-discriminant}
For a~fixed $\lambda\in\R$, let
\begin{equation*}
 \E_\lambda(u)=\lambda\gamma_+(u)+(1-\lambda)\gamma_-(u)
             =\gamma_-(u)+\lambda\Delta_\gamma(u)
\end{equation*}
be the~local affine $\lambda$-equidistant set.  Then
\begin{equation}\label{eq:equidistant-speed}
 \E_\lambda'(u)=
 \bigl(\lambda\rho_+(u)-(1-\lambda)\rho_-(u)\bigr)\mathbbm{e}(u)
 =\bigl(\lambda\Sigma_\rho(u)-\rho_-(u)\bigr)\mathbbm{e}(u).
\end{equation}
Under \eqref{eq:finite-CSS-condition}, $\E_\lambda$ is singular at $u$ if and only if $\lambda=\rho_-(u)/\Sigma_\rho(u)$.  Its singular value is then $S(u)$.  Hence the finite $\CSS$ branch is the~discriminant of the~one-parameter family of affine equidistant sets.  In particular, $\E_{1/2}$ is the~Wigner caustic.
\end{corollary}

\begin{proof}
Differentiate $\E_\lambda$ and use \eqref{eq:common-angle}.  Since $\Sigma_\rho\ne0$ near $0$, the~coefficient in \eqref{eq:equidistant-speed} vanishes precisely at $\lambda=\rho_-/\Sigma_\rho$.  Substitution gives $\E_\lambda=S$ at the corresponding singular point.
\end{proof}

Let $\mathbbm{e}_\Delta=\Delta_\gamma/|\Delta_\gamma|$ and let $\varphi$ be a~local angle of the~chord, so that $\mathbbm{e}_\Delta=(\cos\varphi,\sin\varphi)$.  From \eqref{eq:Delta-gamma-derivative},
\begin{equation}\label{eq:chord-angle}
 \varphi'=\frac{\det(\Delta_\gamma,\Delta_\gamma')}
 {|\Delta_\gamma|^2}
 =\frac{\Sigma_\rho\det(\Delta_\gamma,\mathbbm{e})}
 {|\Delta_\gamma|^2}.
\end{equation}
Conditions \eqref{eq:chord-transverse} and \eqref{eq:finite-CSS-condition} show that $\varphi'(0)\ne0$.  Thus $S$ is a~front with tangent field $\mathbbm{e}_\Delta$.  In the~chord-angle parameter,
\begin{equation}\label{eq:CSS-signed-speed}
 \frac{\dd S}{\dd\varphi}=\sigma\mathbbm{e}_\Delta,
 \qquad
 \sigma=\frac{\mathcal J_\rho|\Delta_\gamma|^3}
 {\Sigma_\rho^3\det(\Delta_\gamma,\mathbbm{e})}.
\end{equation}

The~smooth $5/4$-type of the~$\CSS$ can be read from the~same one-variable data.  Introduce the~differential operator
\begin{equation}\label{eq:chord-angle-operator}
 \mathcal D_\varphi=
 \frac{|\Delta_\gamma|^2}
 {\Sigma_\rho\det(\Delta_\gamma,\mathbbm{e})}
 \frac{\dd}{\dd u}=\frac{\dd}{\dd\varphi}
\end{equation}
and, for $j\geq0$, put
\begin{equation}\label{eq:sigma-jets}
 \sigma_j=(\mathcal D_\varphi^j\sigma)(0).
\end{equation}

For regular curves, the~regular-point criterion and the~ordinary-cusp criterion in the~next theorem were obtained in \cite{GiblinHoltom}.

\begin{theorem}\label{thm:CSS-classification}
Under assumptions \eqref{eq:chord-transverse} and \eqref{eq:finite-CSS-condition}, the~centre symmetry set $S$ is regular at $0$ if and only if $\mathcal J_\rho(0)\ne0$.  Moreover:
\begin{enumerate}[(i)]
\item $S$ is a~$3/2$-cusp if and only if $\mathcal J_\rho(0)=0$ and $\mathcal J_\rho'(0)\ne0$;
\item $S$ is a~$4/3$-cusp if and only if $\mathcal J_\rho(0)=\mathcal J_\rho'(0)=0$ and $\mathcal J_\rho''(0)\ne0$;
\item if $\mathcal J_\rho$, $\mathcal J_\rho'$, and $\mathcal J_\rho''$ vanish at $0$, while $\mathcal J_\rho'''(0)\ne0$, then $S$ is $C^1$-equivalent to a~$5/4$-cusp.
\item Under the~conditions in~(iii), define
      \begin{equation}\label{eq:QS}
      Q_S=35\sigma_4^2-48\sigma_3\sigma_5+2400\sigma_3^2.
      \end{equation}
Then $S$ is a~$5/4$-cusp, a~$(5/4;+7)$-cusp, or a $(5/4;-7)$-cusp according as $Q_S$ is zero, positive, or negative, respectively.
\end{enumerate}
Equivalently, wherever it is defined, singular points of $S$ are the~critical points of the~projective ratio $[\rho_+:\rho_-]\colon(\R,0)\to\RP^1$.
\end{theorem}

\begin{proof}
All factors multiplying $\mathcal J_\rho$ in \eqref{eq:CSS-signed-speed} are nonzero near $0$.  The~order of the~signed speed $\sigma$, with respect to the~regular parameter $\varphi$, is therefore the~order of $\mathcal J_\rho$. The~regularity assertion and items~(i)--(iii) follow from Lemma~\ref{lem:cusp-criteria}.  Under the~conditions in~(iii), the nonvanishing factors in \eqref{eq:CSS-signed-speed} imply $\sigma_0=\sigma_1=\sigma_2=0$ and $\sigma_3\ne0$.  In the~parameter $\varphi$ one has
\begin{equation*}
 \frac{\dd S}{\dd\varphi}=\sigma\mathbbm{e}_\Delta,
 \qquad \mathbbm{e}_\Delta=(\cos\varphi,\sin\varphi).
\end{equation*}
Thus $\sigma$ is the~extended signed radius of curvature of $S$ in its angular parameter, and Lemma~\ref{lem:cusp-criteria}(vi) proves~(iv).

In an~affine chart of $\RP^1$ in which $\rho_-\ne0$, the~derivative of $\rho_+/\rho_-$ is
\begin{equation*}
 \left(\frac{\rho_+}{\rho_-}\right)'
 =-\frac{\mathcal J_\rho}{\rho_-^2}.
\end{equation*}
The~other chart is analogous.
\end{proof}

\begin{proposition}\label{prop:CSS-cuspidal-curvature}
If $S$ has a~$3/2$-cusp at $0$, then
\begin{equation}\label{eq:CSS-cuspidal-curvature}
 \omega_S=
 \frac{2\Sigma_\rho(0)
 \det(\Delta_\gamma(0),\mathbbm{e}(0))}
 {\sqrt{\left|\left(\dfrac{\rho_-}{\Sigma_\rho}\right)''\!(0)\right|}
 \,|\Delta_\gamma(0)|^{5/2}},
 \qquad
 \left(\frac{\rho_-}{\Sigma_\rho}\right)''\!(0)
 =\frac{\mathcal J_\rho'(0)}{\Sigma_\rho(0)^2}.
\end{equation}
\end{proposition}

\begin{proof}
At a~cusp, by Theorem \ref{thm:CSS-classification} we get that $\left(\rho_-/\Sigma_\rho\right)'(0)=0$, and \eqref{eq:CSS-derivative} gives
\begin{equation*}
 S''(0)=
 \left(\frac{\rho_-}{\Sigma_\rho}\right)''\!(0)\Delta_\gamma(0),
 \qquad
 S'''(0)=
 \left(\frac{\rho_-}{\Sigma_\rho}\right)'''\!(0)\Delta_\gamma(0)
 +2\left(\frac{\rho_-}{\Sigma_\rho}\right)''\!(0)
 \Sigma_\rho(0)\mathbbm{e}(0).
\end{equation*}
Therefore
\begin{equation*}
 \det(S''(0),S'''(0))
 =2\left(\frac{\rho_-}{\Sigma_\rho}\right)''\!(0)^2
 \Sigma_\rho(0)\det(\Delta_\gamma(0),\mathbbm{e}(0)).
\end{equation*}
Dividing by $|S''(0)|^{5/2}$ proves the~first formula.  The~second follows from $\left(\rho_-/\Sigma_\rho\right)'=\mathcal J_\rho/\Sigma_\rho^2$ and $\mathcal J_\rho(0)=0$.
\end{proof}

\section{Relations between the~two constructions}
\label{sec:relations}

\noindent The~conditions controlling the~two branches admit a~compact projective interpretation.  The~Wigner caustic is singular when $\Delta_\rho=0$.  If $\Sigma_\rho=0$ and $\rho_+\rho_-\ne0$, the~affine centre symmetry set goes to infinity. Simultaneous zeros require the~separate analysis in Section~\ref{sec:projective-CSS}.  Under \eqref{eq:chord-transverse} and \eqref{eq:finite-CSS-condition}, the~centre symmetry set is singular when the~projective ratio $[\rho_+:\rho_-]$ is critical.  The~following order-lowering principle makes the~relation more precise.

\begin{theorem}\label{thm:order-lowering}
Assume \eqref{eq:chord-transverse} and
\begin{equation}\label{eq:equal-nonzero-radii}
 \rho_+(0)=\rho_-(0)\ne0.
\end{equation}
If $\Delta_\rho$ has a~zero of order $m\geq1$ at $0$, then $\mathcal J_\rho$ has a~zero of order $m-1$.  Furthermore,
\begin{equation}\label{eq:W-equals-S}
 \E_{1/2}(0)=S(0)=\frac{\gamma_+(0)+\gamma_-(0)}2.
\end{equation}
Consequently:
\begin{enumerate}[(i)]
\item a~$3/2$-cusp of $\E_{1/2}$ lies on a~regular part of $S$;
\item a~$4/3$-cusp of $\E_{1/2}$ is a~$3/2$-cusp of $S$;
\item a~$5/4$- or $(5/4;\mathord\pm7)$-cusp of $\E_{1/2}$ is a $4/3$-cusp of $S$.
\end{enumerate}
\end{theorem}

\begin{proof}
Since $\rho_-=\rho_+-\Delta_\rho$,
\begin{equation}\label{eq:J-Delta}
 \mathcal J_\rho
 =-\rho_+\Delta_\rho'+\rho_+'\Delta_\rho.
\end{equation}
If $\ord_0\Delta_\rho=m$, the~first term on the~right-hand side of \eqref{eq:J-Delta} has order $m-1$ and nonzero leading coefficient, whereas the~second has order at least $m$.  Thus $\ord_0\mathcal J_\rho=m-1$.  Under \eqref{eq:equal-nonzero-radii}, the~quotient $\rho_-(0)/\Sigma_\rho(0)$ equals $1/2$, which proves \eqref{eq:W-equals-S}.  The~classification then follows from Theorems~\ref{thm:Wigner-difference} and \ref{thm:CSS-classification}.
\end{proof}

There is an~analogous lowering rule when only one endpoint of the~chord is a~singular point of a~front.

\begin{corollary}\label{cor:source-cusp-lowering}
Assume \eqref{eq:chord-transverse}, $\rho_+(0)=0$, and $\rho_-(0)\ne0$.  If $\rho_+$ has a~zero of order $m$, then
\begin{equation*}
 S(0)=\gamma_+(0),
 \qquad \ord_0\mathcal J_\rho=m-1.
\end{equation*}
In particular, a~$3/2$-cusp of $\gamma_+$ gives a~regular point of $S$, a~$4/3$-cusp gives a~$3/2$-cusp of $S$, and a~front-type $5/4$- or $(5/4;\mathord\pm7)$-cusp gives a~$4/3$-cusp of $S$.
\end{corollary}

\begin{proof}
Formula \eqref{eq:CSS-formula} and $\rho_-(0)/\Sigma_\rho(0)=1$ give $S(0)=\gamma_+(0)$.  The~term $-\rho_-\rho_+'$ in $\mathcal J_\rho$ has order $m-1$ and nonzero leading coefficient, whereas $\rho_+\rho_-'$ has order at least $m$. Remark~\ref{rem:angular-radius} and Theorem~\ref{thm:CSS-classification} complete the~proof.
\end{proof}

We now globalise the~preceding relation for an~oval.  In the~ordinary-cusp case the~result is the~inequality proved in \cite[Theorem~7.2]{DomitrzRiosGCS}. The~finite-order formulation below also covers degenerate cusps.

\begin{theorem}\label{thm:oval-count}
Let $\gamma$ be a~smooth strictly convex oval, parametrised by its tangent angle $u\in\R/2\pi\mathbb Z$,
\begin{equation}\label{eq:oval-radius}
 \gamma'(u)=\rho(u)\mathbbm{e}(u),\qquad \rho(u)>0.
\end{equation}
Define the~anti-periodic function
\begin{equation}\label{eq:oval-log-ratio}
 F(u)=\log\frac{\rho(u)}{\rho(u+\pi)},
 \qquad F(u+\pi)=-F(u).
\end{equation}
Assume that $F$ and $F'$ have only finitely many zeros and that every zero has finite order.  Their zero sets and their orders descend to $\R/\pi\mathbb Z$. Put
\begin{align*}
 \mathcal Z_{\E_{1/2}}
 &=\{[u]\in\R/\pi\mathbb Z:F(u)=0\},&
 \mathcal Z_{\CSS}
 &=\{[u]\in\R/\pi\mathbb Z:F'(u)=0\},\\
 \mathcal O_{\E_{1/2}}
 &=\sum_{[u]\in\mathcal Z_{\E_{1/2}}}\ord_uF,&
 \mathcal O_{\CSS}
 &=\sum_{[u]\in\mathcal Z_{\CSS}}\ord_uF'.
\end{align*}
Then
\begin{equation}\label{eq:oval-count-inequalities}
 \#\mathcal Z_{\CSS}\geq\#\mathcal Z_{\E_{1/2}},
 \qquad \mathcal O_{\CSS}\geq\mathcal O_{\E_{1/2}},
 \qquad
 \mathcal O_{\E_{1/2}}\equiv\mathcal O_{\CSS}\equiv1\pmod 2.
\end{equation}
Moreover, $\mathcal Z_{\E_{1/2}}$ and $\mathcal Z_{\CSS}$ are precisely the~singular-parameter sets of the~Wigner caustic and the~centre symmetry set, respectively, and the~orders occurring in $\mathcal O_{\E_{1/2}}$ and $\mathcal O_{\CSS}$ are the~orders of their respective signed speeds.
\end{theorem}

\begin{proof}
For the~parallel pair $\gamma(u),\gamma(u+\pi)$, set $\rho_+(u)=\rho(u)$ and $\rho_-(u)=\rho(u+\pi)$.  Since both radii are positive, $\Sigma_\rho>0$. Strict convexity also makes the~chord transverse to the~common tangent.  Furthermore,
\begin{equation}\label{eq:F-identities}
 \Delta_\rho
 =\rho(u+\pi)\bigl(\exp(F(u))-1\bigr),
 \qquad
 F'(u)=-\frac{\mathcal J_\rho(u)}
 {\rho(u)\rho(u+\pi)}.
\end{equation}
Thus $F$ has the~same zero set and the~same zero orders as $\Delta_\rho$, while $F'$ has the~same zero set and zero orders as $\mathcal J_\rho$.  Theorems~\ref{thm:Wigner-difference} and \ref{thm:CSS-classification} give the~last assertion.

It remains to prove \eqref{eq:oval-count-inequalities}.  Work first on $\R/2\pi\mathbb Z$ and list the~distinct zeros of $F$ cyclically.  Rolle's theorem supplies a~zero of $F'$ in every open interval between consecutive zeros of $F$ -- these zeros are distinct.  Hence the~number of zeros of $F'$ is at least the~number of zeros of $F$.  If a~zero of $F$ has order $m$, it contributes a~zero of $F'$ of order $m-1$ when $m\geq2$, while the~Rolle zero in the~following interval contributes at least one.  After summing cyclically, the~total zero multiplicity of $F'$ is therefore at least the~total zero multiplicity of $F$.  Anti-periodicity pairs every zero with a~zero of the~same order at distance $\pi$.  Dividing both inequalities by two proves the~first two claims on $\R/\pi\mathbb Z$.

Finally, choose a~base point which is not a~zero of either $F$ or $F'$. Each of these anti-periodic functions changes sign over the~following half-period.  A~finite-order zero changes the~sign precisely when its order is odd.  Therefore the~sum of the~zero orders of each function in a half-period is odd, which proves the~two congruences.
\end{proof}

\begin{corollary}\label{cor:oval-lower-bound}
Under the~hypotheses of Theorem~\ref{thm:oval-count},
\begin{equation*}
 \#\mathcal Z_{\CSS}\geq\#\mathcal Z_{\E_{1/2}}\geq3.
\end{equation*}
\end{corollary}

\begin{proof}
The~Blaschke--S\"uss theorem (see, for instance, \cite{Gozdz,Laugwitz}) gives at least three antipodal pairs on an oval. We recall that a pair of points on an oval is an \emph{antipodal pair} if the tangent lines are parallel at these points and curvatures are equal. These are precisely the~classes in $\mathcal Z_{\E_{1/2}}$.  See \cite{DomitrzZwierSingular} for a~modern account and generalisations. The~first inequality is Theorem~\ref{thm:oval-count}.
\end{proof}

\begin{remark}
At a~zero of $F$ of order $m$, identity \eqref{eq:F-identities} recovers the~local order drop $m\mapsto m-1$ at the~same midpoint, while Rolle's theorem detects the~additional $\CSS$ singularities between consecutive Wigner singularities.  In particular, a~Wigner $4/3$-cusp forces a~CSS $3/2$-cusp at the~same point, and a~Wigner $5/4$-type cusp forces a~CSS $4/3$-cusp there.  When all zeros of $F$ and $F'$ are simple, \eqref{eq:oval-count-inequalities} reduces to the~classical ordinary-cusp inequality.
\end{remark}

\section{Projective completion of the~centre symmetry set}
\label{sec:projective-CSS}

\noindent The~condition $\Sigma_\rho\ne0$ is automatic for a~strictly convex curve with coherent orientation, but not for a~general frontal.  The~radius formula has a~natural projective completion.  We write homogeneous points of $\RP^2$ as $[x:w]$, where $x\in\R^2$ and $w\in\R$.

\begin{proposition}\label{prop:projective-CSS}
The~map
\begin{equation}\label{eq:projective-CSS}
 \widetilde S(u)=
 [\rho_-(u)\gamma_+(u)+\rho_+(u)\gamma_-(u):\Sigma_\rho(u)]
\end{equation}
agrees with the~affine centre symmetry set whenever $\Sigma_\rho\ne0$.  Suppose $\Delta_\gamma(0)\ne0$.
\begin{enumerate}[(i)]
\item If $\Sigma_\rho(0)=0$ and $\rho_+(0)\rho_-(0)\ne0$, then $\widetilde S(0)=[\Delta_\gamma(0):0]$, up to a~nonzero homogeneous factor.  Thus the~$\CSS$ point is at infinity in the~direction of the~chord.
\item Suppose $\rho_+$ and $\rho_-$ have finite positive orders $p$ and $q$, respectively. After cancelling their common factor of order $\min\{p,q\}$, \eqref{eq:projective-CSS} extends to $u=0$.  If $p<q$, the~limit is $[\gamma_-(0):1]$. If $q<p$, it is $[\gamma_+(0):1]$.
\item If $p=q$ and
      \begin{equation*}
      \rho_+(u)=a_+u^p+O(u^{p+1}),\qquad
      \rho_-(u)=a_-u^p+O(u^{p+1}),
      \end{equation*}
with $a_+a_-\ne0$, then the~limit is
      \begin{equation*}
      \left[
      \frac{a_-\gamma_+(0)+a_+\gamma_-(0)}{a_++a_-}:1
      \right]
      \end{equation*}
when $a_++a_-\ne0$, and is $[\Delta_\gamma(0):0]$ when $a_++a_-=0$.
\end{enumerate}
\end{proposition}

\begin{proof}
The~first assertion follows by homogenising \eqref{eq:CSS-formula}.  If $\Sigma_\rho(0)=0$ and both radii are nonzero, then $\rho_-(0)=-\rho_+(0)$ and the~first homogeneous component is a~nonzero multiple of $\gamma_-(0)-\gamma_+(0)=-\Delta_\gamma(0)$.  This proves (i).

For (ii)--(iii), divide all three homogeneous coordinates by $u^{\min\{p,q\}}$.  If $p<q$, the~reduced value of $\rho_-$ vanishes and that of $\rho_+$ does not, so the~limit is $[\gamma_-(0):1]$.  The~case $q<p$ is symmetric.  If $p=q$, the~reduced homogeneous point is
\begin{equation*}
 [a_-\gamma_+(0)+a_+\gamma_-(0):a_++a_-].
\end{equation*}
This is the~stated affine point when $a_++a_-\ne0$.  When $a_++a_-=0$, its first component is a~nonzero multiple of $\Delta_\gamma(0)$.
\end{proof}

\begin{remark}\label{rem:ratio-map}
The~map $[\rho_+:\rho_-]$ simultaneously records four distinguished behaviours:
\begin{equation*}
 [\rho_+:\rho_-]=[1:1] \Longleftrightarrow \E_{1/2}'=0,
 \qquad
 [\rho_+:\rho_-]=[1:-1]
 \Longleftrightarrow \widetilde S\text{ is at infinity},
\end{equation*}
while $[\rho_+:\rho_-]=[0:1]$ and $[1:0]$ place the~affine $\CSS$ point at one of the~two chord endpoints.  Its critical points are precisely the~singular points of the~finite admissible $\CSS$ branch.  The~undefined value $[0:0]$ is resolved by Proposition~\ref{prop:projective-CSS} whenever the~two zeros have finite order.
\end{remark}

\section{Parallel pairs containing a~non-front singularity}

\noindent We now drop angular regularity at the~first endpoint.  Let $\gamma_+$ be a~cooriented frontal germ with Legendre data $(\ell,\theta)$ from \eqref{eq:Legendre-data}.  Let $\gamma_-$ be an~angularly regular cooriented frontal germ parametrised by its tangent angle $u$, and denote its chosen tangent lift by $\mathbbm{e}_-$.  We choose its sign so that $\mathbbm{e}_-(\theta(t))=-\mathbbm{e}(t)$ at corresponding points.  If $\rho_-$ denotes the~extended signed radius of curvature of $\gamma_-$, defined as in \eqref{eq:extended-radius} relative to $\mathbbm{e}_-$, then
\begin{equation}\label{eq:nonfront-minus-branch}
 \gamma_-'(\theta(t))
 =\rho_-(\theta(t))\mathbbm{e}_-(\theta(t))
 =-\rho_-(\theta(t))\mathbbm{e}(t).
\end{equation}

The~corresponding local family of parallel pairs is $t\mapsto\bigl(\gamma_+(t),\gamma_-(\theta(t))\bigr)$, because the~chosen tangent lifts at these two points are opposite.  Put
\begin{equation}\label{eq:degenerate-chord}
 \Delta_\gamma(t)=\gamma_+(t)-\gamma_-(\theta(t)),
 \qquad
 \rho_{-,0}=\rho_-(\theta(0)).
\end{equation}
Thus $\rho_{-,0}$ is the~extended signed radius of curvature of $\gamma_-$ at the~second endpoint of the~base parallel pair.  Notice that $\gamma_-$ is regular at this endpoint if and only if $\rho_{-,0}\ne0$.  We assume
\begin{equation}\label{eq:degenerate-transversality}
 \Delta_\gamma(0)\ne0,\qquad
 \det(\Delta_\gamma(0),\mathbbm{e}(0))\ne0.
\end{equation}

\begin{proposition}\label{prop:degenerate-master}
The~Wigner caustic
\begin{equation*}
 \E_{1/2}(t)=\frac{\gamma_+(t)+\gamma_-(\theta(t))}{2}
\end{equation*}
is a~frontal with the~same tangent field $\mathbbm{e}$ and signed speed
\begin{equation}\label{eq:degenerate-Wigner-speed}
 \ell_{\E_{1/2}}(t)
 =\frac12\bigl(\ell(t)-\rho_-(\theta(t))\theta'(t)\bigr).
\end{equation}
For the~chord family from $\gamma_-(\theta(t))$ to $\gamma_+(t)$, the~finite centre symmetry set, where defined, is
\begin{equation}\label{eq:degenerate-CSS}
 S(t)=\gamma_-(\theta(t))+\Lambda(t)\Delta_\gamma(t),
 \qquad
 \Lambda(t)=
 \frac{\rho_-(\theta(t))\theta'(t)}
 {\ell(t)+\rho_-(\theta(t))\theta'(t)}.
\end{equation}
It satisfies $S'=\Lambda'\Delta_\gamma$.  Its projective extension is
\begin{equation}\label{eq:degenerate-projective-CSS}
 \widetilde S(t)=
 [\rho_-(\theta(t))\theta'(t)\gamma_+(t)
  +\ell(t)\gamma_-(\theta(t)):
  \ell(t)+\rho_-(\theta(t))\theta'(t)].
\end{equation}
\end{proposition}

\begin{proof}
Differentiating the~Wigner midpoint and using \eqref{eq:nonfront-minus-branch} gives \eqref{eq:degenerate-Wigner-speed}.  Moreover,
\begin{equation*}
 \Delta_\gamma'(t)=
 \bigl(\ell(t)+\rho_-(\theta(t))\theta'(t)\bigr)\mathbbm{e}(t).
\end{equation*}
The~envelope condition for the~line $\gamma_-(\theta(t))+\lambda\Delta_\gamma(t)$ gives \eqref{eq:degenerate-CSS}, exactly as in the~proof of Proposition~\ref{prop:CSS-envelope}.  The~terms parallel to $\mathbbm{e}$ cancel after differentiating, and hence $S'=\Lambda'\Delta_\gamma$.  Homogenising \eqref{eq:degenerate-CSS} gives \eqref{eq:degenerate-projective-CSS}.
\end{proof}

We first treat a~$5/2$-cusp, and later a~$5/3$-cusp.

\begin{theorem}\label{thm:five-two-resonances}
Assume that $\gamma_+$ is a~$5/2$-cusp at $0$.  Then:
\begin{enumerate}[(i)]
\item If $\ell'(0)-\rho_{-,0}\theta''(0)\ne0$, the~Wigner caustic is a $5/2$-cusp.  If $\ell'(0)-\rho_{-,0}\theta''(0)=0$, it is a $5/3$-cusp.
\item Suppose $\rho_{-,0}\ne0$ and $\ell'(0)+\rho_{-,0}\theta''(0)\ne0$.  Then the~centre symmetry set extends smoothly across $0$ and is regular there.  More precisely,
      \begin{equation}\label{eq:Lambda-five-two}
      \Lambda(0)=
      \frac{\rho_{-,0}\theta''(0)}{\ell'(0)+\rho_{-,0}\theta''(0)},
      \qquad
      \Lambda'(0)=
      \frac{\rho_{-,0}
      \bigl(\ell'(0)\theta'''(0)-\ell''(0)\theta''(0)\bigr)}
      {2\bigl(\ell'(0)+\rho_{-,0}\theta''(0)\bigr)^2}\ne0.
      \end{equation}
\item Suppose $\rho_{-,0}\ne0$ and $\ell'(0)+\rho_{-,0}\theta''(0)=0$.  Then the projective centre symmetry set tends to $[\Delta_\gamma(0):0]$.  Thus the~affine branch goes to infinity in the~chord direction.
\end{enumerate}
If $\theta''(0)\ne0$, we call the~value of $\rho_{-,0}$ producing the~exceptional Wigner case in~(i) the~\emph{Wigner-resonant signed radius}, and the~value producing the~projective degeneration in~(iii) the~\emph{$\CSS$-resonant signed radius}.  These two radii are opposite:
\begin{equation}\label{eq:opposite-resonances}
 \rho_{-,0}^{\,\E_{1/2}}=\frac{\ell'(0)}{\theta''(0)},
 \qquad
 \rho_{-,0}^{\,\CSS}=-\frac{\ell'(0)}{\theta''(0)}.
\end{equation}
\end{theorem}

\begin{proof}
For the~signed speed $\ell_{\E_{1/2}}$ in \eqref{eq:degenerate-Wigner-speed}, the~vanishing of $\theta'(0)$ gives
\begin{equation}\label{eq:ellW-derivatives}
 \ell_{\E_{1/2}}'(0)=\frac{\ell'(0)-\rho_{-,0}\theta''(0)}{2},
 \qquad
 \ell_{\E_{1/2}}''(0)=\frac{\ell''(0)-\rho_{-,0}\theta'''(0)}2.
\end{equation}
Therefore
\begin{equation}\label{eq:invariant-cancellation}
 \ell_{\E_{1/2}}'(0)\theta'''(0)
 -\ell_{\E_{1/2}}''(0)\theta''(0)
 =\frac{\ell'(0)\theta'''(0)-\ell''(0)\theta''(0)}{2}\ne0.
\end{equation}
If $\ell'(0)-\rho_{-,0}\theta''(0)\ne0$, Lemma~\ref{lem:cusp-criteria}(iii) gives a~$5/2$-cusp.  If equality holds, then $\theta''(0)\ne0$ because $\ell'(0)\ne0$, and \eqref{eq:invariant-cancellation} becomes
\begin{equation*}
 \ell_{\E_{1/2}}''(0)\theta''(0)
 =-\frac{\ell'(0)\theta'''(0)-\ell''(0)\theta''(0)}{2}\ne0.
\end{equation*}
Lemma~\ref{lem:cusp-criteria}(iv) gives a~$5/3$-cusp.

For the~$\CSS$ part, set $\eta(t)=\rho_-(\theta(t))\theta'(t)$.  Since $\theta'(0)=0$,
\begin{equation*}
 \eta(0)=0,\qquad
 \eta'(0)=\rho_{-,0}\theta''(0),\qquad
 \eta''(0)=\rho_{-,0}\theta'''(0).
\end{equation*}
Write $\ell(t)=t\bar\ell(t)$ and $\eta(t)=t\bar\eta(t)$.  If $\ell'(0)+\rho_{-,0}\theta''(0)\ne0$, then
\begin{equation*}
 \Lambda(t)=\frac{\bar\eta(t)}{\bar\ell(t)+\bar\eta(t)}
\end{equation*}
is smooth.  Its value and derivative at $0$ are exactly \eqref{eq:Lambda-five-two}.  Since $\rho_{-,0}\bigl(\ell'(0)\theta'''(0)-\ell''(0)\theta''(0)\bigr)\ne0$ and $\Delta_\gamma(0)\ne0$, $S'(0)=\Lambda'(0)\Delta_\gamma(0)\ne0$.

If $\ell'(0)+\rho_{-,0}\theta''(0)=0$, divide the~homogeneous coordinates in \eqref{eq:degenerate-projective-CSS} by $t$.  Their value at $0$ is
\begin{equation*}
 [\rho_{-,0}\theta''(0)\gamma_+(0)+\ell'(0)\gamma_-(\theta(0)):
   \ell'(0)+\rho_{-,0}\theta''(0)]
 =[\ell'(0)(\gamma_-(\theta(0))-\gamma_+(0)):0],
\end{equation*}
which is the~point at infinity in the~direction of $\Delta_\gamma(0)$.  Formula \eqref{eq:opposite-resonances} is now immediate.
\end{proof}

\begin{remark}\label{rem:bias-resonance}
The~bias of a~$5/2$-cusp in the~data \eqref{eq:Legendre-data} is
\begin{equation*}
 b_{\gamma_+}=\frac{3\theta''(0)}{|\ell'(0)|}
\end{equation*}
(see \cite{Teramoto}).  Hence, when $b_{\gamma_+}\ne0$, both resonant radii in \eqref{eq:opposite-resonances} have absolute value $3/|b_{\gamma_+}|$. If the~bias vanishes, neither finite resonance is possible: the~Wigner caustic is a~$5/2$-cusp and the~$\CSS$ branch is regular for every $\rho_{-,0}\ne0$.
\end{remark}

\begin{theorem}\label{thm:five-three-transfer}
Assume that $\gamma_+$ is a~$5/3$-cusp at $0$.
\begin{enumerate}[(i)]
\item If $\rho_{-,0}\ne0$, the~Wigner caustic is a~$5/2$-cusp.  The~centre symmetry set extends smoothly, passes through $\gamma_+(0)$, and is regular. More precisely,
      \begin{equation}\label{eq:Lambda-five-three}
      \Lambda(0)=1,
      \qquad
      \Lambda'(0)=-
      \frac{\ell''(0)}{2\rho_{-,0}\theta''(0)}\ne0.
      \end{equation}
\item If $\rho_{-,0}=0$, the~Wigner caustic is a~$5/3$-cusp.
\end{enumerate}
\end{theorem}

\begin{proof}
Let $\ell_{\E_{1/2}}$ be given by \eqref{eq:degenerate-Wigner-speed}.  Since $\ell(0)=\ell'(0)=\theta'(0)=0$, one has
\begin{equation*}
 \ell_{\E_{1/2}}'(0)=-\frac{\rho_{-,0}\theta''(0)}2,
 \qquad
 \ell_{\E_{1/2}}''(0)=\frac{\ell''(0)-\rho_{-,0}\theta'''(0)}2.
\end{equation*}
If $\rho_{-,0}\ne0$, then $\ell_{\E_{1/2}}'(0)\ne0$ and
\begin{equation*}
 \ell_{\E_{1/2}}'(0)\theta'''(0)
 -\ell_{\E_{1/2}}''(0)\theta''(0)
 =-\frac{\ell''(0)\theta''(0)}2\ne0.
\end{equation*}
Lemma~\ref{lem:cusp-criteria}(iii) gives a~$5/2$-cusp.  If $\rho_{-,0}=0$, then
\begin{equation*}
 \ell_{\E_{1/2}}'(0)=0,
 \qquad
 \ell_{\E_{1/2}}''(0)\theta''(0)
 =\frac{\ell''(0)\theta''(0)}2\ne0,
\end{equation*}
so Lemma~\ref{lem:cusp-criteria}(iv) gives a~$5/3$-cusp.

Assume now that $\rho_{-,0}\ne0$.  In \eqref{eq:degenerate-CSS}, the~function $\eta=\rho_-(\theta)\theta'$ has a~simple zero, while $\ell$ has a~double zero.  Thus $\Lambda=\eta/(\eta+\ell)$ extends with $\Lambda(0)=1$. Expanding
\begin{equation*}
 \eta(t)=\rho_{-,0}\theta''(0)t+O(t^2),
 \qquad
 \ell(t)=\frac{\ell''(0)}2t^2+O(t^3),
\end{equation*}
gives \eqref{eq:Lambda-five-three}.  Since $S'=\Lambda'\Delta_\gamma$, the~$\CSS$ branch is regular and $S(0)=\gamma_+(0)$.
\end{proof}

\section*{Statements and Declarations}

\noindent\textbf{Funding.} The~author received no financial support for the~research, authorship, or publication of this article.

\noindent\textbf{Competing interests.}
The~author declares that he has no competing interests.

\noindent\textbf{Data availability.} No datasets were generated or analysed in this theoretical study.

\noindent\textbf{Use of generative artificial intelligence.} During the~preparation of this manuscript, the~author used OpenAI's ChatGPT (GPT-5.6) as an auxiliary tool to assist with language editing and checking mathematical derivations and computations. All AI-assisted content was independently reviewed and verified by the~author, who takes full responsibility for the~manuscript.

\end{document}